\documentclass{article} 

\usepackage[all,cmtip]{xy}
\usepackage{amsmath}
\usepackage{amssymb}
\newtheorem{theorem}{Theorem}

\newtheorem{lemma}{Lemma}

\newtheorem{corollary}{Corollary}
\newenvironment{proof}{\paragraph{\ Proof:}}{\hfill$\square$}

\newcommand{\R}{\mathbb{R}}
\newcommand{\Z}{\mathbb{Z}}

\title{On a Conjecture of D. B. Shapiro} 

\author{ Jun Zhu \\
        \small junzhu1277@gmail.com \\        
}

\date{\today} 

\begin{document} 
    
	\maketitle 
    
    \section{Abstract} 
		We prove a conjecture stated in [4], which asserts that no $[10,10,16]_\Z$ formula 
	can arise as a restriction of any Hurwitz$-$Radon formula. Consequently, the unique 
	$[10,10,16]_\Z$ formula provides the first known example of a composition formula that cannot be 
	obtained from classical Hurwitz–Radon formulas by a process of restrictions and direct sums.

	\section{Introduction} 
		Given integers $r, s$ and $n$, a composition or a formula of size $[r,s,n]_\R$ is a square 
	identity of type
	\begin{equation}
	(x_1^2 + \cdots + x_r^2)(y_1^2 +\cdots + y_s^2) = z_1^2 + \cdots + z_n^2 
	\end{equation}
	where $x=(x_1, ..., x_r)$ and $y=(y_1, ..., y_s)$ are systems of indeterminates and each 
	$z_k = z_k(x,y)$ is a bilinear form in $x$ and $y$ with coefficients in $\R$. Given $r, s$, 
	$r*s$ is the smallest number of n such that there exists an $[r,s,n]_\R$ formula. Such a 
	formula is equivalent to a normed bilinear map $f: R^r \times R^s \rightarrow R^n$ satisfying
	\begin{equation}
	|f(x,y)|=|x| |y|, x \in \R^r, y \in \R^s
	\end{equation}
	
	Two $[r,s,n]_\R$ formulas are said to be equivalent if their associated normed bilinear maps $f,g$
	differ by orthogonal changes of coordinates, namely the following diagram is commutative:
	
	\[
	\xymatrix{
	\R^r \times \R^s \ar[r]^{g} \ar[d]_{\alpha \times \beta} & \R^n \ar[d]^{\gamma } \\
	\R^r \times \R^s \ar[r]^{f} & \R^n
	}
	\]
	where $\alpha, \beta$ and $\gamma$ are isometries.
	
	Replace $\R$ with $\Z$ (integers), then the $[r,s,n]_\Z$ formula is over integers and 
	each $z_k(x,y)$ is a bilinear form in x and y with coefficients in $\Z$. Given $r,s$, $r *_\Z s$ is the 
	smallest number of n such that there exists an $[r,s,n]_{\Z}$ formula. Since an $[r,s,n]_\Z$ formula 
	is also an $[r,s,n]_\R$ formula, we have $r *_\Z s \ge r *_\R s$.
	
	A matrix $M$ of size $r \times s$ is an intercalate matrix if:
	\begin{enumerate}
	\item All entries are nonnegative integers (called colors).	
	\item The colors along each row (resp. column) are distinct.	
	\item If $M(i, j ) = M(i',j')$ then $M(i, j') = M(i',j )$.(intercalacy)
	\end{enumerate}

	An intercalate matrix M is consistently signed if there exist $\epsilon_{ij} = \pm 1$ such that
	$\epsilon_{ij} \epsilon_{ij'} \epsilon_{i'j} \epsilon_{i'j'} = -1$ 
	whenever $M(i, j ) = M(i',j')$ and $i \ne i'$ and $j \ne j'$.

	It is well known that there exists an $[r,s,n]_\Z$ formula if and only if there is a consistently signed
	$r\times s$ intercalate matrix with $n$ colors. 
	
	Given a consistently signed intercalate matrix $M$, 
	the corresponding normed bilinear map $f=(z_1,z_2,...,z_n)$ can be defined as follows:
	each $z_k$ is determined by color $k$, if color $k$ appears in position $(i,j)$ with a sign $c=\pm 1$,
    then $z_k$ has a term $cx_iy_j$. For example, if $M$ is the following matrix:
	\[
	M = \left[
	\begin{array}{rrrrrrrrrr}
	1 &  2 &  3 &  4  \\
	2 & -1 &  4 & -3  \\
	3 & -4 & -1 &  2  \\
	4 &  3 & -2 & -1  
	
	\end{array}
	\right]
	\]
	
	then the normed bilinear map is
	\[
	f(x, y) = \left(
	\begin{aligned}
	z_1 &= x_1y_1 - x_2y_2 - x_3y_3 - x_4y_4 \\
	z_2 &= x_1y_2 + x_2y_1 + x_3y_4 - x_4y_3 \\
	z_3 &= x_1y_3 - x_2y_4 + x_3y_1 + x_4y_2 \\
	z_4 &= x_1y_4 + x_2y_3 - x_3y_2 + x_4y_1	
	\end{aligned}
	\right)
	\]
	
	\begin{lemma}
		If $f=(z_1,z_2,...,z_n)$ is the normed bilinear map of an $[r, s, n]_\Z$ formula, then every term $x_iy_j$ of the map $f$ has coefficient $\pm 1$ and only appears in one $z_k$.
	\end{lemma}
	\begin{proof}
		Pick two vectors $x=(0,…,a_i,0,…,0)\in \R^r$, $y=(0,…,b_j,0,…,0)\in \R^s$, 
		where $a_i=1$ at the ith position and $b_j=1$ at the jth position. 
		Then $1=|x|^2|y|^2=|f(x,y)|^2=z_1(x,y)^2+...+z_n(x,y)^2$, since $f$ has integer coefficients, 
		if $x_iy_j$ appears in $z_k$, then $z_k(x,y)^2= c^2\ge 1$, where $c$ is the coefficient of $x_iy_j$, so $c=\pm 1$ and all other $z_k(x,y)^2=0$,
		if follows that $x_iy_j$ has coefficient $\pm 1$ and only appears in one $z_k$.
	 
	\end{proof}	
	
	A classical result of Hurwitz and Radon states that an $[r, n, n]_\R$ formula exists if and only if $r \leq \rho(n)$, 
	where $\rho(n)$ is the Hurwitz$-$Radon function deﬁned as follows: if $n = 2^{4a+b}n_0$ where 
	$n_0$ is odd and $0 \leq b \leq 3$, then $\rho(n) = 8a + 2b$. A $[\rho(n), n, n]_\R$ formula is called a Hurwitz$-$Radon formula.
	From [2] and [4], we know that for every $[r, s, n]_\R$ with $n - r \leq 5$, 
	there is a formula built from the classical Hurwitz$-$Radon formulas 
	by a process of restrictions and direct sums. A natural question is whether the statement 
	remains true for $n-r>5$, it is difficult to answer. Even for the smallest case $[10,10,16]$, 
	we know there exists a $[10,10,16]_\Z$ formula that is unique and not a direct sum of 
	any other formulas, however, we do not know if it is a restriction of some 
	Hurwitz$-$Radon formulas, Shapiro conjectured in [4] that no $[10,10,16]_\Z$ formula can be a 
	restriction of a Hurwitz$-$Radon formula in 2000. In this paper, we are going to prove 
	this conjecture. Hence the $[10,10,16]_\Z$ formula is the first example that cannot be built 
	from the classical Hurwitz$-$Radon formulas by a process of restrictions and direct sums. 
	
	\section{Restrictions of Hurwitz$-$Radon formulas}
		Let $F$ be the normed bilinear map of a Hurwitz$-$Radon formula $[h,m,m]_\R$, a restriction of 
	the Hurwitz$-$Radon formula is an $[r,s,n]_\R$ formula with $r \leq h, s \leq m, n \leq m$ 
	such that there exist subspaces $X\subseteq\mathbb R^h, Y\subseteq\mathbb R^m$, and $Z\subseteq\mathbb R^m$ of dimensions 
	$r,s$, and $n$, respectively, where $Z$ contains the 
	image of $X \times Y$ under F, and the normed bilinear map of the $[r,s,n]_\R$ formula is the restriction 
	of $F$ on $X \times Y$, namely, the normed bilinear map of the $[r,s,n]_\R$ formula $f: X \times Y \rightarrow Z$ 
	is defined as $f(x,y)=F(x,y)$.
	
	If a restriction $[r,s,n]_\R$ formula is equivalent to an integral formula $[r,s,n]_\Z$ with bilinear map $g$, 
	then there are isometries 
	$\alpha :\R^r \rightarrow X, \beta :\R^s \rightarrow Y, \gamma : \R^n \rightarrow Z$
	such that the following diagram is commutative: 
	\[
	\xymatrix{
	\R^r \times \R^s \ar[r]^{g} \ar[d]_{\alpha \times \beta} & \R^n \ar[d]^{\gamma } \\
	X \times Y \ar[r]^{f} & Z
	}
	\]

	Images of standard orthonormal bases of $\R^r$, $\R^s$ and $\R^n$ are orthonormal bases of 
	$X$,$Y$ and $Z$. Clearly, these bases can be extended to orthonormal bases for 
	$\R^h$, $\R^m$ and $\R^m$ so that the first r-dimensional subspace of $\R^h$ is $X$, 
	the first s-dimensional subspace of $\R^m$ is $Y$ and the first n-dimensional subspace of 
	$\R^m$ is $Z$. The restriction $F$ on $X \times Y$ is the normed bilinear map $g$.
	
	We use $x=(x_1,…,x_r,…,x_h)$, $y=(y_1,…,y_s, …,y_m)$ and $z=(z_1,…,z_n,…,z_m)$ 
	as vectors of $\R^h$, $\R^m$ and $\R^m$, where $x=(x_1,…,x_r)$, $y=(y_1,…,y_s)$ and $z=(z_1,…,z_n)$ 
	as vectors of $X$, $Y$ and $Z$ and every $z_k=z_k(x,y)$ is a bilinear form and the normed bilinear map $f$ 
	can be expressed as $f(x,y)=g(x,y)=(z_1,...,z_n)$ by removing all the terms of $cx_iy_j$ with $i>r$ or $j>s$. 
	
	We say $x_iy_j$ appears in $z_k$ if $z_k$ has a term $cx_iy_j$ with $c\ne 0$. 
			
	\begin{lemma}
		Bilinear forms $z_1$,…, $z_n$ have the following properties: 
		
		\textnormal{(1)}	If $x_iy_j$ with $i \leq r$, $j \leq s$ appears in some $z_k$, then $x_iy_j$ does not appear in any other $z_{k'}$.
		
		\textnormal{(2)}	If $x_iy_j$ with $i\leq r, j\leq s$ and $x_{i’}y_{j’}$ with $i’\leq r$ or $j’\leq s$
		appear in $z_k$, then $x_iy_{j’}$ and $x_{i’}y_j$ must appear in a unique $z_{k'}$. Further more, 
		if coefficients of $x_iy_j, x_{i’}y_{j’}, x_iy_{j’}, x_{i’}y_j$ are $c_{ij},c_{i'j'},c_{ij'},c_{i'j}$ in $z_k$ and $z_{k'}$, 
		then $c_{ij}c_{i'j'}=-c_{ij'}c_{i'j}$.
		
		\textnormal{(3)}	Every $z_k$ has a term $x_iy_j$ for every $x_i$. 

	\end{lemma}
	\begin{proof}
	(1) Since $i \leq r,j \leq s, x_iy_j$, by Lemma 1, it only appears in one $z_k$ with $k\leq n$ and has coefficient $\pm 1$.
	Let $a_i=b_j=1$, $x=(0,…,a_i,0,…,0)\in \R^h$, $y=(0,…,b_j,0,…,0)\in \R^m$, 
	then $1=|x|^2|y|^2=|F(x,y)|^2=z_1(x,y)^2+… +z_m(x,y)^2$, since $z_k(x,y)^2=1$, no other $z_{k'}(x,y)^2>0$, namely, $x_iy_j$ does not 
	appear in any other $z_k$.
	
	(2) $z_k^2$ contains a term $2c_{ij}c_{i'j'}x_iy_jx_{i’}y_{j’}$, by the formula identity, this term is cancelled 
	from terms of other $z_{k'}^2$. By (1), $x_iy_j$ appears only in $z_k$, hence, there must exist a $z_{k'}$ that contains 
	both $x_iy_{j’}$ and $x_{i’}y_j$. If $i’\leq r$, since $j\leq s$, $x_{i’}y_j$ appears only in one $z_{k'}$ by (1), 
	no other $z_{k''}$ contains these two terms. To cancel the term $2c_{ij}c_{i'j'}x_iy_jx_{i’}y_{j’}$, coefficients of 
	these terms must satisfy $c_{ij}c_{i'j'} = -c_{ij'}c_{i'j}$.
	If $j'<s$, we can prove the same result with a similar argument.

	(3) Let $x =(0,…,x_i,0,…,0)$ with $x_i=1$ be a basis vector in $\R^h$, since F is a normed bilinear map, 
	$F$ defines a map $F_x: \R^m \rightarrow \R^m$: $y \rightarrow F(x,y)$ that is isometric, it follows that 
	the image of $F_x$ has dimension $m$. Hence there must be some vector $y$, $F_x(y)$ is not zero in $z_k$. 	
	this means there exists some $y_j$ such that $x_iy_j$ appears in $z_k$.
		
	\end{proof}
			
	\vspace{2mm}	
	
	\section{Proof of the conjecture}
		For $[10,10,16]_\Z$, Yiu proved in [6] that every $[10,10,16]_\Z$ is equivalent to the formula with the following consistently signed intercalate matrix:
		\vspace{2mm}
	    
	\[
	\left[
	\begin{array}{rrrrrrrrrr}
	1 &  2 &  3 &  4 &  5 &  6 &  7 &  8 &  9 &  10 \\
	2 & -1 &  4 & -3 &  6 & -5 & -8 &  7 & 10 &  -9 \\
	3 & -4 & -1 &  2 &  7 &  8 & -5 & -6 & 11 &  12 \\
	4 &  3 & -2 & -1 &  8 & -7 &  6 & -5 & 12 & -11 \\
	5 & -6 & -7 & -8 & -1 &  2 &  3 &  4 & 13 &  14 \\
	6 &  5 & -8 &  7 & -2 & -1 & -4 &  3 & 14 & -13 \\
	7 &  8 &  5 & -6 & -3 &  4 & -1 & -2 & 15 & -16 \\
	8 & -7 &  6 &  5 & -4 & -3 &  2 & -1 & 16 &  15 \\
	9 & -10& -11& -12& -13& -14& -15& -16& -1 &   2 \\
	10&  9 & -12&  11& -14&  13&  16& -15& -2 &  -1 
	\end{array}
	\right]
	\]

	Let $x=(x_1,x_2,x_3,x_4,x_5,x_6,x_7,x_8,x_9,x_{10})$, $y=(y_1,y_2,y_3,y_4,y_5,y_6,y_7,y_8,y_9,y_{10})$ and $z=(z_1,z_2,z_3 ,z_4,z_5,z_6,z_7,z_8,z_9,z_{10},z_{11},z_{12},z_{13},z_{14},z_{15},z_{16})$ be points in $R^{10}$, $R^{10}$ and $R^{16}$.
	Then the corresponding bilinear map $f: \R^{10} \times \R^{10} \rightarrow \R^{16}$ defined as 
	\vspace{4mm}	
    
	\[
	f(x, y) = \left(
	\begin{aligned}
	z_1 &= x_1y_1 - x_2y_2 - x_3y_3 - x_4y_4 - x_5y_5 - x_6y_6 - x_7y_7 - x_8y_8 - x_9y_9 - x_{10}y_{10} \\
	z_2 &= x_1y_2 + x_2y_1 + x_3y_4 - x_4y_3 + x_5y_6 - x_6y_5 - x_7y_8 + x_8y_7 + x_9y_{10} - x_{10}y_9 \\
	z_3 &= x_1y_3 - x_2y_4 + x_3y_1 + x_4y_2 + x_5y_7 + x_6y_8 - x_7y_5 - x_8y_6 \\
	z_4 &= x_1y_4 + x_2y_3 - x_3y_2 + x_4y_1 + x_5y_8 - x_6y_7 + x_7y_6 - x_8y_5 \\
	z_5 &= x_1y_5 - x_2y_6 - x_3y_7 - x_4y_8 + x_5y_1 + x_6y_2 + x_7y_3 + x_8y_4 \\
	z_6 &= x_1y_6 + x_2y_5 - x_3y_8 + x_4y_7 - x_5y_2 + x_6y_1 - x_7y_4 + x_8y_3 \\
	z_7 &= x_1y_7 + x_2y_8 + x_3y_5 - x_4y_6 - x_5y_3 + x_6y_4 + x_7y_1 - x_8y_2 \\
	z_8 &= x_1y_8 - x_2y_7 + x_3y_6 + x_4y_5 - x_5y_4 - x_6y_3 + x_7y_2 + x_8y_1 \\
	z_9 &= x_1y_9 - x_2y_{10} + x_9y_1 + x_{10}y_2 \\
	z_{10} &= x_1y_{10} + x_2y_9 - x_9y_2 + x_{10}y_1 \\
	z_{11} &= x_3y_9 - x_4y_{10} - x_9y_3 + x_{10}y_4 \\
	z_{12} &= x_3y_{10} + x_4y_9 - x_9y_4 - x_{10}y_3 \\
	z_{13} &= x_5y_9 - x_6y_{10} - x_9y_5 + x_{10}y_6 \\
	z_{14} &= x_5y_{10} + x_6y_9 - x_9y_6 - x_{10}y_5 \\
	z_{15} &= x_7y_9 + x_8y_{10} - x_9y_7 - x_{10}y_8 \\
	z_{16} &= -x_7y_{10} + x_8y_9 - x_9y_8 + x_{10}y_7 
	\end{aligned}
	\right)
	\]
	
	\begin{theorem} The normed bilinear map of a $[10,10,16]_\Z$ formula is not a 
	restriction of any normed bilinear map of size $[h,m,m]$.
	\end{theorem}
	\begin{proof} Let $F=(z_1,z_2,…,z_{16},z_{17},…,z_m)$ be normed bilinear map of 
	size $[h,m,m]$, if the normed bilinear map of a $[10,10,16]_\Z$ formula is a 
	restriction of $F$, as in the previous section, we can assume $F$ restricts to 
	$\R^{10} \times \R^{10}\rightarrow \R^{16}$ is the normed bilinear map defined above, 
	then every bilinear map $z_i$ of $[h,m,m]$ formula with $i \leq 16$ 
	contains terms listed in the above table with respect to $z_i$.
	
	By Lemma 2 (3), every bilinear map $z_i$ of $[h,m,m]$ formula has a term starting with $x_9$, 
	hence, for example, $z_3$ must have a term $cx_9y_j$ with $j>10$ because every $x_9y_j$ 
	with $j \leq 10$ must already appear in the above table. Without loss of generality, 
	we can assume $y_j=y_{11}$, then $z_3$ of the restriction of $F$ on $\R^{10} \times \R^{11}$, 
	$z_3$ has a term $cx_9y_{11}$. Since $x_1y_3$ is in $z_3$ and does not appear in any other $z_k$, 
	so $x_1y_{11}$ and $x_9y_3$ appear in a unique $z_k$ by Lemma 2 (2).  From the above table, 
	$x_9y_3$ appear in $z_{11}$ and does not appear in any other $z_i$ by Lemma 2 (1), it follows 
	that $x_1y_{11}$ must appear in $z_{11}$ as well and the coefficient is also $c$. Similarly, 
	since $-x_2y_3$ is in $z_3, -x_2y_{11}$ must appear in $z_{12}$, similarly for other terms in $z_3$
	as well, we can show that $z_9,z_{10},…,z_{16}$ have terms: 
	$-cx_3y_{11}, cx_4y_{11}, cx_1y_{11}, -cx_2y_{11}, -cx_7y_{11}, -cx_8y_{11}, cx_5y_{11}, cx_6y_{11}$ 
	respectively:
	\vspace{4mm}	
		
	\begin{tabular}{ll}		
		$z_9 = x_1y_9	-x_2y_{10}	+x_9y_1	+x_{10}y_2	-cx_3y_{11}	$				\\
		$z_{10} = x_1y_{10}	+x_2y_9	-x_9y_2	+x_{10}y_1	+cx_4y_{11}$				\\	
		$z_{11} = x_3y_9	-x_4y_{10}	-x_9y_3	+x_{10}y_4	+cx_1y_{11}	$				\\
		$z_{12} = x_3y_{10}	+x_4y_9	-x_9y_4	-x_{10}y_3	-cx_2y_{11}$				\\	
		$z_{13} = x_5y_9	-x_6y_{10}	-x_9y_5	+x_{10}y_6	-cx_7y_{11}$				\\
		$z_{14} = x_5y_{10}	+x_6y_9	-x_9y_6	-x_{10}y_5	-cx_8y_{11}$				\\	
		$z_{15} = x_7y_9	+x_8y_{10}	-x_9y_7	-x_{10}y_8	+cx_5y_{11}	$				\\
		$z_{16} = -x_7y_{10}	+x_8y_9	-x_9y_8	+x_{10}y_7	+cx_6y_{11}$

    \end{tabular}
    
	\vspace{4mm}	
	
	For every $z_k$ with $k\geq 9$, there is a new term $x_iy_{11}$, by the above argument, some new terms should
	appear in other $z_k$. For example, for $z_{11}$, we have the following:
	\vspace{4mm}
		
	\begin{tabular}{ll}		
		$z_{3} = x_1y_3	-x_2y_4	+x_3y_1	+x_4y_2	+x_5y_7	+x_6y_8	-x_7y_5	-x_8y_6	+cx_9y_{11}	$				\\
		$z_{4} = x_1y_4	+x_2y_3	-x_3y_2	+x_4y_1	+x_5y_8	-x_6y_7	+x_7y_6	-x_8y_5	-cx_{10}y_{11}$				\\
		$z_9 = x_1y_9	-x_2y_{10}	+x_9y_1	+x_{10}y_2	-cx_3y_{11}	$				\\
		$z_{10} = x_1y_{10}	+x_2y_9	-x_9y_2	+x_{10}y_1	+cx_4y_{11}$					
			
    \end{tabular}
    
	\vspace{4mm}
	
	Now, $z_{4}$ has a term $x_{10}y_{11}$, applying the above argument for $z_3$, we get the following:
	\vspace{4mm}
		
	\begin{tabular}{ll}		
		$z_{9} = x_1y_9	-x_2y_{10}	+x_9y_1	+x_{10}y_2	+cx_3y_{11}	$				\\
		$z_{10} = x_1y_{10}	+x_2y_9	-x_9y_2	+x_{10}y_1	-cx_4y_{11}$				\\
		$z_{11} = x_3y_9	-x_4y_{10}	-x_9y_3	+x_{10}y_4	-cx_1y_{11}	$				\\
		$z_{12} = x_3y_{10}	+x_4y_9	-x_9y_4	-x_{10}y_3	+cx_2y_{11}$				\\	
		$z_{13} = x_5y_9	-x_6y_{10}	-x_9y_5	+x_{10}y_6	-cx_7y_{11}	$				\\
		$z_{14} = x_5y_{10}	+x_6y_9	-x_9y_6	-x_{10}y_5	-cx_8y_{11}$				\\
		$z_{15} = x_7y_9	+x_8y_{10}	-x_9y_7	-x_{10}y_8	+cx_5y_{11}$				\\
		$z_{16} = -x_7y_{10}	+x_8y_9	-x_9y_8	+x_{10}y_7	+cx_6y_{11}$					
		
    \end{tabular}
    
	\vspace{4mm}	
	
	Notice that $z_{11}$ that has a term $-cx_1y_{11}$, but $z_{11}$ has a term $cx_1y_{11}$ in the previous table. 
	This sign difference means $c=0$ that is a contradiction. This completes the proof.
	\end{proof}
	\vspace{4mm}
	
	Since every Hurwitz$-$Radon formula $[h,m,m]_\R$ or $[h,m,m]_\Z$ has a normed bilinear map, Theorem 1 
	implies the following:
	\begin{corollary}
	No $[10,10,16]_\Z$ formula is a restriction of any Hurwitz$-$Radon formula.
	\end{corollary}
	
	\section{Summary}
			We proved the conjecture of D. B. Shapiro, hence, not every $[r,s,n]_\R$ or $[r,s,n]_\Z$ formula can be
	    built from Hurwitz$-$Radon formulas by a process of restrictions and direct sums. 
		The $[10,10,16]_\Z$ formula has a very interesting structure, it is not just the smallest formula that is not a restriction 
		of any Hurwitz$-$Radon formulas but also the largest case for $ r*r = r\circ r$, where $r\circ r$
		is the smallest integer $n$ such that there exists an $r\times r$ intercalate matrix with $n$ colors.
		
		For a consistently signed intercalate matrix, a submatrix is also a consistently signed intercalate matrix and its 
		normed bilinear map is a restriction of the original normed bilinear map. Hence, by Theorem 1,
		a consistently signed intercalate matrix of a Hurwitz$-$Radon formula $[h,m,m]_\Z$ 
		has no submatrices of type $[10,10,16]_\Z$.
		
		It is also not difficult to create more $[r,s,n]_\Z$ formulas with $n-s>6$ that is not a restriction of 
		any Hurwitz$-$Radon formula by extending $[10,10,16]_\Z$ formula. For example, the consistently signed intercalate 
		matrix of a $[12,12,26]_\Z$ formula in [5, (7.2)] has a submatrix of type $[10,10,16]_\Z$, 
		so it cannot be a restriction of any Hurwitz$-$Radon formula.
	
	\section*{Acknowledgements}

	The author would like to thank Professor K. Y. Lam for his valuable comments and encouragement.

\end{document}